\documentclass[11pt]{article}
\usepackage[T1]{fontenc}
\usepackage[utf8]{inputenc}
\usepackage{lmodern}

\usepackage{hyperref}
\usepackage{graphicx}
\usepackage[english]{babel}
\usepackage{amsfonts,amssymb,mathrsfs,amscd,amsmath,amsthm}

\newtheorem{theorem}{Theorem}
\newtheorem{lemma}{Lemma}

\newtheorem{statement}{Statement} 
\newtheorem{conjecture}{Conjecture}
\theoremstyle{definition}
\newtheorem{example}{Example} 
\newtheorem{definition}{Definition} 
\newtheorem{remark}{Remark}

\newcommand{\freeprod}{\mathop{\ast}\limits}

\def\R{{\mathbb R}}

\def\Z{{\mathbb Z}}

\title{\Huge \textbf{Word and Conjugacy Problems in Groups $G_{k+1}^{k}$}}
\author{\textsc{Denis A. Fedoseev\footnote{Lomonosov Moscow State University and Trapeznikov ICS RAS}, Andrey B. Karpov\footnote{Lomonosov Moscow State University},} \\ \textsc{Vassily O. Manturov}\footnote{Bauman Moscow State Technical University and Novosibirsk State University}}

\begin{document}

\maketitle

\begin{abstract}
Recently the third named author defined a 2-parametric family of groups $G_n^k$ \cite{gnk}. Those groups may be regarded as a certain generalisation of braid groups. Study of the connection between the groups $G_n^k$ and dynamical systems led to the discovery of the following fundamental principle: ``If dynamical systems describing the motion of $n$ particles possess a nice codimension one property governed by exactly $k$ particles, then these dynamical systems admit a topological invariant valued in $G_{n}^{k}$''.

The $G_n^k$ groups have connections to different algebraic structures, Coxeter groups and Kirillov-Fomin algebras, to name just a few. Study of the $G_n^k$ groups led to, in particular, the construction of invariants, valued in free products of cyclic groups.

In the present paper we prove that word and conjugacy problems for certain $G_{k+1}^k$ groups are algorithmically solvable.
\end{abstract}

\vspace{5mm}



\section{Introduction}

The groups $G_{n}^{k}$ were introduced by the third named author into the study of configuration spaces, manifolds, and dynamical systems. 

In \cite{gnk}, the following principle was declared: 

\begin{center}
{\em If dynamical systems describing the motion of $n$ particles possess a nice codimension 1 property governed by exactly $k$ particles then they possess a topological invariant valued in $G_{n}^{k}$.} \end{center} 

The groups $G_{n}^{k}$ are given by a presentation with three types of relations: all generators are involutions, there is far commutativity relation and the triangle relation (see Definition \ref{def:gnk} below).

For $G_{n}^{k}$ groups are closely related to braid groups, it is natural to pose the word problem and the conjugacy problem for them. The connection of the groups $G_n^k$ with fundamental groups of manifolds, with Coxeter groups and other geometrical and algebraic structures was widely studied (see, for example, \cite{KimMan, Coxeter, HigherGnk, Man}, see also \cite{FedManCheng}, where braids with points were introduced, which are closely related to the $G_n^2$ group).

The class of groups for $n=k+1$ is especially interesting because of the absence of the far commutativity relation. The present paper is devoted to the study of groups $G_{k+1}^{k}$. \\

First, let us recall the definition of the {\em group $G_n^k$}, see \cite{gnk}.

\begin{definition}
The {\em group $G_{n}^{k}$} is defined as the quotient group of the free group generated by all $a_{m}$ for all multiindices $m$  by relations~\eqref{eq:gnk_tetrahedron_relation},~\eqref{eq:gnk_far_commutativity_relation} and \eqref{eq:gnk_order2_relation}, see below.
\label{def:gnk}
\end{definition}

Consider the following $n \choose k$ generators $a_{m}$, where $m$ runs the set of all unordered $k$-tuples $m_{1},\dots, m_{k},$ whereas each $m_{i}$ are pairwise distinct numbers from $\left\{1,\dots, n\right\}$.

For each $(k+1)$-tuple $U$ of indices $u_{1},\dots, u_{k+1} \in \{1,\dots, n\}$, consider the $k+1$ sets $m^{j}=U\setminus \{u_{j}\}, j=1,\dots, k+1$. With $U$, we associate the relation

\begin{equation}\label{eq:gnk_tetrahedron_relation}
a_{m^1}\cdot a_{m^2}\cdots a_{m^{k+1}}= a_{m^{k+1}}\cdots a_{m^2}\cdot a_{m^1}
\end{equation} 
for two tuples $U$ and ${\bar U}$, which differ by order reversal, we get the same relation.

Thus, we totally have
$\frac{(k+1)! {n \choose k+1}}{2}$ relations.

We shall call them the {\em tetrahedron relations}.

For $k$-tuples  $m,m'$ with $Card(m\cap m')<k-1$, consider the {\em far commutativity relation:}

\begin{equation}\label{eq:gnk_far_commutativity_relation}
a_{m}a_{m'}=a_{m'}a_{m}.
\end{equation}

Note that the far commutativity relation can occur only if $n>k+1$.

Besides that, for all multiindices $m$, we write down the following relation:

\begin{equation}\label{eq:gnk_order2_relation}
a_{m}^{2}=1.
\end{equation}

\begin{example}
The group $G_{3}^{2}$ is $\langle a,b,c\,|\,a^{2}=b^{2}=c^2=(abc)^{2}=1\rangle,$ where $a=a_{12},b=a_{13},c=a_{23}$.

Indeed, the relation $(abc)^{2}=1$ is equivalent to the relation $abc=cba=1$ because of the relations $a^{2}=b^{2}=c^{2}=1$. This obviously yields all the other tetrahedron relations.
\end{example}

\begin{example} 
The group $G_{4}^{3}$ is isomorphic to $\langle a,b,c,d\,|\,a^{2}=b^{2}=c^{2}=d^2=1,(abcd)^{2}=1,(acdb)^{2}=1,(adbc)^{2}=1\rangle$. Here $a=a_{123},b=a_{124},c=a_{134},d=a_{234}.$

It is easy to check that instead of $\frac{4!}{2}=12$ relations, it suffices to take only $\frac{3!}{2}$ relations.
\end{example}

In the present paper we concentrate on the study of the groups $G_{k+1}^k$. \\

The paper is organized as follows. In Section \ref{sec:realisation_n=k+1} we discuss the geometric interpretation of the groups $G_{k+1}^k$. In the third section we prove the solvability of the conjugacy problem in $G_4^3$. The last section is devoted to the study of the $G_5^4$ case. The solvability of the word problem in $G_5^4$ is proved. \\

The first named author was supported by the program ``Leading Scientific Schools'' (grant no. NSh-6399.2018.1, Agreement No. 075-02-2018-867) and by the Russian Foundation for Basic Research (grant No. 19-01-00775-a). The third named author was supported by the Laboratory of Topology and Dynamics, Novosibirsk State University (grant No. 14.Y26.31.0025 of the government of the Russian Federation). \\

The authors are grateful to A.A. Klyachko and I.M. Nikonov for useful discussions.

\section{Realisation of the groups $G_{k+1}^k$}
\label{sec:realisation_n=k+1}

In the present section we tackle the problem of geometric interpretation of the groups $G_n^k$ for the special case $n=k+1$.

More precisely, we shall prove the following
\begin{theorem}
There is a subgroup ${\tilde G}_{k+1}^{k}$ of the group $G_{k+1}^{k}$ of index $2^{k-1}$ which is isomorphic to $\pi_{1}({\tilde C}'_{k+1}(\R{}P^{k-1}))$, where the space ${\tilde C}'_{n}$ and the group ${\tilde G}_{k+1}^{k}$ will be defined later in this section.
\label{th0}
\end{theorem}

The simplest case of the above Theorem is

\begin{theorem}
The group ${\tilde G}_{4}^{3}$ (which is a finite index subgroup of $G_4^3$) is isomorphic to $\pi_{1}(FBr_{4}(\R P^{2}))$, the $4$-strand braid group on $\R{}P^{2}$ with two points fixed.\label{th1}
\end{theorem}

Theorem \ref{th1} gives a solution to the word problem in the group $G_4^3$.

\subsection{Necessary definitions}

In the present section we introduce all the necessary notions. First, we give the definition of the {\em restricted configuration spaces} $C'_{n}(\R{}P^{k-1})$ and maps from the corresponding fundamental groups to the groups $G_{n}^{k}$.

Let us fix a pair of natural number $n>k$. A point in $C'_{n}(\R{}P^{k-1})$ is an ordered set of $n$ pairwise distinct points in  $\R{}P^{k-1}$, such that any $(k-1)$ of them are in general position. Thus, for instance, if $k=3$, then the only condition is that these points are pairwise distinct. For $k=4$ for points $x_{1},\dots, x_{n}$ in $\R{}P^{3}$ we impose the condition that no three of them belong to the same line (though some four are allowed to belong to the same plane), and for  $k=5$ a point in $C'_{n}(\R{}P^{4})$ is a set of ordered $n$ points in  $\R{}P^{4}$, with no four of them belonging to the same $2$-plane.

Let us use the projective coordinates $(a_{1}:a_{2}:\cdots: a_{k})$ in $\R{}P^{k-1}$ and let us fix the following  $k-1$ points in general position, $y_{1},y_{2}, \cdots, y_{k-1}\in \R{}P^{k-1}$, where $a_{i}(y_{j})=\delta_{i}^{j}$. Let us define the subspace ${\tilde C}'_{n}(\R{}P^{k-1})$ taking those $n$-tuples of points $x_{1},\cdots, x_{n}\in \R{}P^{k-1}$ for which $x_{i}=y_{i}$ for $i=1,\cdots, k-1$.

We say that a point $x\in C'_{n}(\R^{k-1})$  is {\em singular}, if the set of points $x=(x_{1},\dots, x_{n})$, representing $x$, contains some subset of $k$ points lying on the same $(k-2)$-plane. Let us fix two  non-singular points $x,x'\in C'_{n}(\R{}P^{k-1}).$

We shall consider smooth paths $\gamma_{x,x'}: [0,1]\to C'_{n}(\R^{k-1})$.
For each such path there are values $t$ such that for $\gamma_{x,x'}(t)$ some $k$ of the points belong to the same $(k-2)$-plane. We call these values $t\in [0,1]$ {\em singular}.

On the path $\gamma$, we say that the moment $t$ of passing through the singular point $x$, corresponding to the set  $x_{i_1},\cdots, x_{i_k}$, is {\em transverse} (or stable) if for any sufficiently small perturbation ${\tilde \gamma}$ of the path  $\gamma$, in the neighbourhood of the moment $t$ there exists exactly one time moment $t'$ corresponding to some set of points $x_{i_1},\cdots, x_{i_k}$ non in general position.

\begin{definition}
We say that a path is {\em good and stable} if the following holds:

\begin{enumerate}

\item
The set of singular values $t$ is finite;

\item For every singular value  $t=t_{l}$ corresponding to $n$ points representing $\gamma_{x,x'}(t_{l})$, there exists only one subset of $k$ points belonging to a $(k-2)$-plane;

\item Each singular value is  {\em stable}.

\end{enumerate}
\end{definition}

\begin{definition}
We say that the path without singular values is {\em void}.
\end{definition}

We shall call such paths {\em braids}. To be precise, we give the following definition.

\begin{definition}
A path from $x$ to $x'$ is called a {\em braid} if the points representing $x$ are the same as those representing $x'$ (possibly, in different orders); if $x$ coincides with $x'$ with respect to order, then such a braid is called {\em pure}.
\end{definition}

We say that two  braids $\gamma,\gamma'$ with endpoints $x,x'$ are {\em isotopic} if there exists a continuous family  $\gamma^{s}_{x,x'},s \in [0,1]$ of smooth paths with fixed ends such that  $\gamma^{0}_{x,x'}=\gamma,\gamma^{1}_{x,x'}=\gamma'$. By a small perturbation, any path can be made good and stable (if endpoints are generic, we may also require that the endpoints remain fixed).

There is an obvious concatenation structure on the set of braids: for paths $\gamma_{x,x'}$ and $\gamma'_{x',x''}$, the concatenation is defined as a path $\gamma''_{x,x''}$ such that $\gamma''(t)=\gamma(2t)$ for $t\in [0,\frac{1}{2}]$ and $\gamma''(t)=\gamma'(2t-1)$ for $t\in [\frac{1}{2},1]$; this path can be smoothed in the neighbourhood of  $t=\frac{1}{2}$; the result of such smoothing is well defined up to isotopy.

Thus, the sets of braids and pure braids (for a fixed $x$) admit a group structure. This latter group is certainly isomorphic to the fundamental group  $\pi_{1}(C'_{n}(\R^{k-1}))$. The former group is isomorphic to the fundamental group of the quotient space by the action of the permutation group.

\subsection{The realisability of $G_{k+1}^k$}

The main idea of the proof of Theorem \ref{th0} is to associate with every word in $G_{k+1}^{k}$ a braid in ${\tilde C}'_{k+1}(\R{}P^{k-1})$.

Let us start with the main construction from \cite{HigherGnk}.

With each good and stable braid from $PB_{n}(\R{}P^{2})$ we associate an element of the group  $G_{n}^{k}$ as follows. We enumerate all singular values of our path $0<t_{1}<\dots <t_{l}<1$ (we assume than  $0$ and $1$ are not singular). For each singular value $t_{p}$ we have a set $m_{p}$ of $k$ indices corresponding to the numbers of points which are not in general position. With this value we associate the letter  $a_{m_{p}}$. With the whole path  $\gamma$ (braid) we associate the product $f(\gamma)=a_{m_{1}}\cdots a_{m_{l}}$.

\begin{theorem}\cite{HigherGnk}
The map   $f$ takes isotopic braids to equal elements of the group  $G_{n}^{k}$. For pure braids, the map   $f$ is a homomorphism  $f:\pi_{1}(C'_{n}(\R{}P^{2}))\to G_{n}^{3}$.
\label{thgn3}
\end{theorem}

Now we claim that \\

{\em Every word from  $G_{k+1}^{k}$ can be realised by a path of the above form.} \\

Note that if we replace $\R{}P^{k-1}$ with $\R^{k-1}$, the corresponding statement will fail. Namely, starting with the configuration of four points, $x_{i},i=1,\cdots, 4$, where $x_{1},x_{2},x_{3}$ form a triangle and $x_{4}$ lies inside this triangle, we see that any path starting from this configuration will lead to a word starting from $a_{124},a_{134}, a_{234}$ but not from $a_{123}$. In some sense the point $4$ is ``locked'' and the points are not in the same position.

\begin{figure}
\centering\includegraphics[width=100pt]{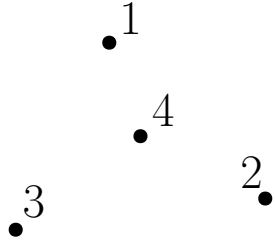}
\caption{The ``locked'' position for the move $a_{123}$}
\label{locked}
\end{figure}

The following well known theorem (see, e.g., \cite{Wu}) plays a crucial role in the construction
\begin{theorem}
For any two sets of $k+1$ points in general position in $\R{}P^{k-1}$, $(x_{1},\cdots, x_{k+1})$ and $(y_{1},\cdots, y_{k+1})$ there is an action of $PGL(k, \R)$ taking all $x_{i}$ to $y_{i}$.\label{wulem}
\end{theorem}

For us, this  will mean that there is no difference between all possible ``non-degenerate starting positions'' for $k+1$ points in $\R{}P^{k}$.

We shall deal with paths in ${\tilde C'}_{k+1}(\R{}P^{k-1})$ similar to braids. Namely, we shall fix a set of $2^{k-1}$ points such that all paths will start and end at these points.

We shall denote homogeneous coordinates in $\R{}P^{k-1}$ by $(a_{1}:\cdots: a_{k})$ in contrast to points (which we denote by $(x_{1},\cdots, x_{k+1})$).

\subsection{Constructing a braid from a word in $G_{k+1}^{k}$}

Our main goal is to construct a braid by a word. To this end, we need a base point for the braid. For the sake of convenience, we shall use not one, but rather $2^{k-1}$ reference points. For the first $k$ points $y_{1}=(1:0:\cdots:0),\cdots, y_{k}=(0:\cdots:0:1)$ fixed, we will have $2^{k-1}$ possibilities for the choice of the last point. Namely, let us consider all possible strings  of length $k$ of $\pm 1$ with the last coordinate $+1$: $$(1,1,\cdots, 1,1),(1,\cdots, 1,-1,1),\cdots, (-1,-1,\cdots, -1,1)$$ with $a_{k} = +1$. We shall denote these points by $y_{s}$ where $s$ records the first $(k-1)$ coordinates of the point.

Now, for each string $s$ of length $k$ of $\pm 1$, we set $z_{s}=(y_{1},y_{2},\cdots, y_{k},y_{s})$.

The following lemma is evident.
\begin{lemma}
For every point $z\in \R{}P^{k-1}$ with projective coordinates $(a_{1}(z):\cdots: a_{k-1}(z):1)$, let ${\tilde z}=(sign(a_{1}(z)):sign(a_{2}(z)): \cdots: sign(a_{k-1}(z)):1)$. Then there is a path between $(y_{1},\cdots, y_{k},z)$ and $(y_{1},\cdots, y_{k}, {\tilde z})$ in ${\tilde C}'_{k+1}(\R{}P^{k})$ with the first points $y_{1},\cdots, y_{k}$ fixed, and the corresponding path in ${\tilde C}'_{k+1}$ is void.
\label{lemmaB}
\end{lemma}

\begin{proof}
Indeed, it suffices just to connect $z$ to ${\tilde z}$ by a geodesic.
\end{proof}

From this we easily deduce the following
\begin{lemma}
Every  point $y\in {\tilde C}'_{k+1}(\R{}P^{k-1})$ can be connected by a void path to some $(y_{1},\cdots, y_{k}, y_{s})$ for some $s$.\label{lemBB}
\end{lemma}

\begin{proof}
Indeed, the void path can be constructed in two steps. At the first step, we construct a path which moves both $y_{k}$ and $y_{k+1}$, so that $y_{k}$ becomes $(0:\cdots 0:1)$, and at the second step, we use Lemma \ref{lemmaB}. To realise the first step, we just use linear maps which keep the hyperplane $a_{k}=0$ fixed.
\end{proof}

The lemma below shows that the path mentioned in Lemma \ref{lemBB} is unique up to homotopy.

\begin{lemma}
Let $\gamma$ be a closed path in ${\tilde C}'_{k+1}(\R{}P^{k-1})$ such that
the word $f(\gamma)$ is empty. Then $\gamma$ is homotopic to the trivial braid.
\label{lmA}
\end{lemma}

\begin{proof}

In ${\tilde C}'_{k+1}$, we deal with the motion of points, where all but $x_{k},x_{k+1}$ are fixed.

Consider the projective hyperplane ${\cal P}_{1}$ passing through $x_{1},\cdots, x_{k-1}$ given by the equation $a_{k}=0$.

We know that none of the points $x_{k}, x_{k+1}$ is allowed to belong to ${\cal P}_{1}$. Hence, we may fix the last coordinate $a_{k}(x_{k})=a_{k}(x_{k+1})=1$.

Now, we may pass to the affine coordinates of these two points (still to be denoted by $a_{1},\cdots, a_{k}$).

Now, the condition $\forall i=1,\cdots, k-1: a_{j}(x_{k})\neq a_{j}(x_{k+1})$ follows from the fact that the points $x_{1},\cdots, {\hat{x_{j} }}, \cdots, x_{k+1}$ are generic.

This means that $\forall i=1,\cdots, k-1$ the sign of  $a_{i}(x_{k})-a_{i}(x_{k+1})$ remains fixed.

Now, the motion of points $x_{k},x_{k+1}$ is determined by their coordinates $a_{1},\cdots, a_{j}$, and since their signs are fixed, the configuration space for this motions is simply connected.

This means that the loop $\gamma$ is described by a loop in a two-dimensional simply connected space.
\end{proof}

Our next strategy is as follows. Having a word in $G_{k+1}^{k}$, we shall associate with this word a path in ${\tilde C}'_{k+1}(\R{}P^{k-1})$. After each letter, we shall get to  $(y_{1},\cdots ,y_{k},y_{s})$ for some $s$.

Let us start from $s=(1,\cdots, 1)$, that is from the set $(y_{1},\cdots, y_{k},y_{1,\cdots, 1})$.

After making the final step, one can calculate the coordinate of the $(k+1)$-th points. They will be governed by Lemma \ref{governed} (see ahead). As we shall see later, those words we have to check for the solution of the word problem in $G_{k+1}^{k}$, will lead us to closed paths, i.e., pure braids.

Let us be more detailed.

\begin{lemma}
Let a non-singular set of points $y$ in $\R{}P^{k-1}$ be given. Then for every set of $k$ numbers $i_{1},i_{2},\cdots, i_{k}\in [n]$, there exists a path $y_{i_{1}\cdots i_{k}}(t)= y(t)$ in $C'_{n}(\R{}P^{k-1})$, having   $y(0)=y(1)=y$ as the starting point and the final point and with only one singular moment corresponding to the numbers $i_{1},\cdots, i_{k}$ which encode the points not in general position; moreover, we may assume that at this moment all points except $i_{1}$, are fixed during the path.

Moreover, the set of paths possessing this property is connected: any other path ${\tilde y}(t)$, possessing the above properties, is homotopic to $y(t)$ in this class.
 \label{lm2}
\end{lemma}

\begin{proof}
Indeed, for the first statement of the Lemma, it suffices to construct a path for some initial position of points and then apply Theorem \ref{wulem}.

For the second statement, let us take two different paths $\gamma_{1}$ and $\gamma_{2}$ satisfying the conditions of the Lemma. By a small perturbation, we may assume that for both of them, $t=\frac{1}{2}$ is a singular moment with the same position of $y_{i_{1}}$.

 Now, we can contract the loop formed by $\gamma_{1}|_{t\in [\frac{1}{2},1]}$ and the inverse of $\gamma_{2}|_{t\in [\frac{1}{2},1]}$ by using Lemma \ref{lmA} as this is a small perturbation of a void braid. We are left with $\gamma_{1}|_{t\in [0,\frac{1}{2}]}$ and the inverse of $\gamma_{2}|_{t\in [0,\frac{1}{2}]}$ which is contractible by Lemma \ref{lmA} again.

\end{proof}

\begin{remark}
Note that in the above lemma, we deal with the space $C'_{n}(\R{}P^{k-1})$, not with ${\tilde C}'_{n}(\R{}P^{k-1})$. On the other hand, we may always choose $i_{1}\in \{k,k+1\}$; hence, the path in question can be chosen in ${\tilde C}'(\R{}P^{k-1})$.
\end{remark}

Now, for every subset $m\subset [n], Card(m)=k+1$ we can create a path $p_{m}$ starting from any of the base points listed above and ending at the corresponding basepoints.

Now, we construct our path step-by step by applying Lemma \ref{lm2} and returning to some of base points by using Lemma \ref{lemBB}.

From \cite{HigherGnk}, we can easily get the following
\begin{lemma}
Let $i_{1},\cdots, i_{k+1}$ be some permutation of $1,\cdots, k+1$. Then the concatenation of paths $p_{i_{1}i_{2}\cdots i_{k}}p_{i_{1}i_{3}i_{4}\cdots i_{k+1}}\cdots p_{i_{2}i_{3}\cdots i_{k}}$ \\ is homotopic to the concatenation of paths in the inverse order $$p_{i_{2}i_{3}\cdots i_{k}}\cdots p_{i_{1}i_{3}i_{4}\cdots i_{k+1}}p_{i_{1}i_{2}\cdots i_{k}}.$$ \label{quadrisec}
\end{lemma}

\begin{proof}
Indeed, in \cite{HigherGnk}, some homotopy corresponding to the above mentioned relation corresponding to {\em some} permuation is discussed. However, since all basepoints are similar to each other as discussed above, we can transform the homotopy from \cite{HigherGnk} to the homotopy for any permutation.
\end{proof}

\begin{lemma}
For the path starting from the point $(y_{1},\cdots, y_{k},y_{s})$ constructed as in  Lemma \ref{lm2} for the set of indices $j$, the endpoints of this path $(y_{1},\cdots, y_{k},y_{s'})$ are such that:
\begin{enumerate}
\item if $j=1,\cdots, k$, then $s'$ differs from $s$ only in coordinate $a_{j}$;
\item if $j=k+1$, all coordinates of $s'$ differ from those coordinates of $s$ by sign.
\end{enumerate}
\label{governed}
\end{lemma}

Denote the map from words in $G_{k+1}^{k}$ to paths between basepoints by $g$.

By construction, we see that for every word $w$ we have $f(g(w))=w\in G_{k+1}^{k}$.

Now, we define the group ${\tilde G}_{k+1}^{k}$ as the subgroup of $G_{k+1}$ which is taken by $g$ to braids. From Lemma \ref{governed}, we see that this is a subgroup of index $(k-1)$: there are exactly $(k-1)$ coordinates. \\

Let us pass to the proof of Theorem \ref{th1}. Our next goal is to see that equal words can originate only from homotopic paths.

To this end, we shall first show that the map $f$ from Theorem \ref{thgn3} is an isomorphism for $n=k+1$. To perform this goal, we should construct the inverse map $g:{\tilde G}_{k+1}^{k}\to \pi_{1}({\tilde C'}_{k+1}(\R{}P(k-1)))$.

Note that for $k=3$ we deal with the pure braids $PB_{4}(\R{}P^{2}).$

Let us fix a point $x\in C'_{4}(\R{}P^{2})$. With each generator $a_{m},m\subset [n],Card(m)=k$ we associate a path  $g(m)=y_{m}(t)$, described in Lemma \ref{lm2}. This path is not a braid: we can return to any of the $2^{k-1}$ base points. However, once we take the concatenation of paths correspoding to ${\tilde G}_{k+1}^{k}$, we get a braid.

By definition of the map $f$, we have $f(g(a_{m}))=a_{m}$. Thus, we have chosen that the map $f$ is a surjection.

Now, let us prove that the kernel of the map $f$ is trivial. Indeed, assume there is a pure braid $\gamma$ such that $f(\gamma)=1\in G_{k+1}^{k}$. We assume that $\gamma$ is good and stable. If this path has $l$ critical points, then we have the word corresponding to it  $a_{m_1}\cdots a_{m_l}\in G_{k+1}^{k}$.

Let us perform the transformation $f(\gamma)\to 1$ by applying the relations of $G_{k+1}^{k}$ to it and transforming the path $\gamma$ respectively. For each relation of the sort $a_{m}a_{m}=1$ for a generator $a_{m}$ of the group $G_{k+1}^{k}$, we see that the path $\gamma$ contains two segments whose concatenation is homotopic to the trivial loop (as follows from the second statement of Lemma \ref{lm2}).

Whenever we have a relation of length $2k+2$ in the group $G_{k+1}$, we use the Lemma \ref{quadrisec} to perform the homotopy of the loops.

 Thus, we have proved that if the word $f(\gamma)$ corresponding to a braid $\gamma\in G_{k+1}^{k}$ is equal to $1$ in $G_{k+1}^{k}$ then the braid $\gamma$ is isotopic to a braid $\gamma'$ such that the word corresponding to it is empty. Now, by Lemma \ref{lmA}, this braid is homotopic to the trivial braid.

\subsection{The group $H_{k}$ and the algebraic lemma}
\label{sec:h_k_and_geomtery}

The aim of the present section is to reduce the word problem in $G_{k+1}^{k}$ to the word problem in a certain subgroup of it, denoted by $H_{k}$.

Let us rename all generators of $G_{k+1}^{k}$ lexicographically:
$$b_{1}=a_{1,2,\cdots,k},\cdots, b_{k+1}=a_{2,3,\cdots, k+1}.$$

Let $H_{k}$ be the subgroup of $G_{k+1}^{k}$ consisting of all elements $x\in G_{k+1}^{k}$ that can be represented by words with no occurencies of the last letter $b_{k+1}$.

Our task is to understand whether a word in $G_{k+1}^{k}$ represents an element in $H_{k}$. That question shall be called the {\em generalised word problem for $H_k$ in $G_{k+1}^k$}. In general, let us introduce the following definition.

\begin{definition}
	Let $G$ be a group and $H\subset G$ be its subgroup. The problem of algorithmic determination whether a given word $w\in G$ belongs to the subgroup $H$ is called the {\em generalised word problem for $H$ in $G$}.
\end{definition}

To solve this problem, we recall the map defined in \cite{MN}. Consider the group $F_{k-1}=\Z_{2}^{*2^{k-1}}=\langle c_{m}|c_{m}^{2}=1\rangle$, where all generators $c_{m}$ are indexed by $(k-1)$-element strings of $0$ and $1$ with only relations being that the square of each generator is equal to $1$. We shall construct a map\footnote{This map becomes a homomorphism when restricted to a finite index subgroup.} from $G_{k+1}^{k}$ to $F_{k-1}$ as follows.

Take a word $w$ in generators of $G_{k+1}^{k}$ and list all occurencies of the last letter $b_{k+1}=a_{2,\cdots, k+1}$ in this word. With any such occurence we first associate the string of indices $0,1$ of length $k$. The $j$-th index is the number of letters $b_{j}$ preceding this occurence of $b_{k+1}$ modulo $2$. Thus, we get a string of length $k$ for each occurence.

Let us consider ``opposite pairs'' of strings $(x_{1},\cdots, x_{k})\sim(x_{1}+1,\cdots, x_{k}+1)$ as equal. Now, we may think that the last ($k$-th) element of our string is always $0$, so, we can restrict ourselves with $(x_{1},\cdots, x_{k-1},0)$. Such a string of length $k-1$ is called the {\em index} of the occurence of $b_{k+1}$.

Having this done, we associate with each occurence of $b_{k+1}$ having index $m$ the generator $c_{m}$ of $F_{k-1}$. With the word $w$, we associate the word $f(w)$ equal to the product of all generators $c_{m}$ in order.

In \cite{MN}, the following Lemma is proved:
\begin{lemma}
The map $f:G_{k+1}^{k}\to F_{k-1}$ is well defined.
\end{lemma}

Now, let us prove the following crucial
\begin{lemma}
If $f(w)=1$ then $w\in H_{k}$.
\end{lemma}

In other words, the free group $F_{k-1}$ yields the only obstruction for an element from $G_{k+1}^{k}$ to have a presentation with no occurence of the last letter.

\begin{proof}
Let $w$ be a word such that $f(w)=1$. If $f(w)$ is empty, then there is nothing to prove. Otherwise $w$ contains two ``adjacent'' occurencies of the same index. This means that $w=A b_{k+1} B b_{k+1} C$, where $A$ and $C$ are some words, and $B$ contains no occurencies of $b_{k+1}$ and the number of occurencies of $b_{1},b_{2},\cdots, b_{k}$ in $B$ are of the same parity.

Our aim is to show that $w$ is equal to a word with smaller number of $b_{k+1}$ in $G_{k+1}^{k}$. Then we will be able to induct on the number of $b_{k+1}$ until we get a word without $b_{k+1}$.

Thus, it suffices for us to show that $b_{k+1}Bb_{k+1}$ is equal to a word from $H_{k}$. We shall induct on the length of $B$. Without loss of generality, we may assume that $B$ is reduced, i.e., it does not contain adjacent $b_{j}b_{j}$.

Let us read the word $B$ from the very beginning $B=b_{i_1}b_{i_2}\cdots$ If all elements $i_{1},i_{2},\cdots$ are distinct, then, having in mind that the number of occurencies of all generators in $B$ should be of the same parity, we conclude that $b_{k+1}B= B^{-1} b_{k+1}$, hence $b_{k+1}Bb_{k+1}=B^{-1}b_{k+1}b_{k+1}=B^{-1}$ is a word without occurencies of $b_{k+1}$.

Now assume ${i_1}={i_p}$ (the situation when the first repetition is for ${i_j}={i_p},1<j<p$ is handled in the same way). Then we have $b_{k+1}B=b_{k+1}b_{i_1}\cdots b_{i_{p-1}} b_{i_1} B'$. Now we collect all indices distinct from $i_{1},\cdots, i_{p-1},{k+1}$ and write the word $P$ containing exactly one generator for each of these indices (the order does not matter). Then the word $W = P b_{k+1} b_{i_{1}}\cdots b_{i_{p-1}}$ contains any letter exactly once and we can invert the word $W$ as follows: $W^{-1}=b_{i_{p-1}}\cdots b_{i_{1}}b_{k+1}P^{-1}$. Thus, $b_{k+1}B=P^{-1}(Pb_{k+1}b_{i_1}\cdots b_{i_{p-1}})b_{i_1}B'=P^{-1}b_{i_{p-1}}\cdots b_{i_1}b_{k+1}P^{-1}b_{i_1}B'$.

We know that the letters in $P$ (hence, those in $P^{-1}$) do not contain $b_{i_1}$. Thus, the word $P^{-1} b_{i_1}$ consists of distinct letters. Now we perform the same trick: create the word $Q=b_{i_2}b_{i_3}\cdots b_{i_{p-1}}$ consisting of remaining letters from $\{1,\cdots, k\}$, we get:

$$ P^{-1}b_{i_{p-1}}\cdots b_{i_1}b_{k+1}P^{-1}b_{i_1}B'$$
$$=P^{-1}b_{i_{p-1}}\cdots b_{i_1}QQ^{-1}b_{k+1}P^{-1}b_{i_1}B'$$
$$=P^{-1}b_{i_{p-1}}\cdots b_{i_1}Qb_{i_1}Pb_{k+1}QB'.$$

Thus, we have moved $b_{k+1}$ to the right and between the two occurencies of the letter $b_{k+1}$, we replaced $b_{i_1}\cdots, b_{i_{p-1}}b_{i_1}$ with just $b_{i_2}\cdots b_{i_{p-1}}$, thus, we have shortened the distance between the two adjacent occurencies of $b_{k+1}$.

Arguing as above, we will finally cancel these two letters $b_{k+1}$ and perform the induction step.
\end{proof}

\begin{theorem}
Generalised word problem for $H_k$ in $G_{k+1}^k$ is solvable.\label{kl}
\end{theorem}

\begin{proof}
Indeed, having a word $w$ in generators of $G_{k+1}^{k}$, we can look at the image of this word by the map $f$. If it is not equal to $1$, then, from \cite{MN}, it follows that $w$ is non-trivial, otherwise we can construct a word ${\tilde w}$ in $H_{k}$ equal to $w$ in $G_{k+1}^{k}$.
\end{proof}

\begin{remark}
A shorter proof of Theorem \ref{kl} based on the same ideas was communicated to the authors by A.A.~Klyachko. We take the subgroup $K_{k}$ of $G_{k+1}^{k}$ generated by products $B_{\sigma}=b_{\sigma_{1}}\cdots b_{\sigma_{k}}$ for all permutations $\sigma$ of $k$ indices. This group $K_{k}$ contains the commutator of $H_{k}$ and is a normal subgroup in $G_{k+1}^{k}$. 

Moreover, the quotient group $G_{k+1}^{k}/ K_{k}$ is naturally isomorphic to the free product $(H_{k}/K_{k})*\langle b_{k+1}\rangle$. Hence, the problem whether an element of $G_{k+1}^{k}/ K_{k}$ belongs to $H_{k}/ K_{k}$ is trivially solved, which solves the initial problem because of the normality of $K_{k}$ in $G_{k+1}^{k}$.
\end{remark}

Certainly, to be able to solve the word problem in $H_{k}$, one needs to know a presentation for $H_{k}$. It is natural to take $b_{1},\cdots, b_{k}$ for generators of $H_{k}$. Obviously, they satisfy the relations $b_{j}^{2}=1$ for every $j$.

To understand the remaining relations for different $k$, we shall need geometrical arguments. \\

We have completely constructed the isomorphism between the (finite index subgroup) of the group $G_{k+1}^{k}$ and a fundamental group of some configuration space.

This completely solves the word problem in $G_{4}^{3}$ for braid groups in projective spaces are very well studied. The same can be said about the conjugacy problem in ${\tilde G}_{4}^{3}$.

Besides, we have seen that the word problem for the case of general $G_{k+1}^{k}$
can be reduced to the case of $H_{k}$.

It would be very interesting to compare the approach of $G_{n}^{k}$ with various generalizations of braid groups, e.g., Manin-Schechtmann groups \cite{ManinSchechtmann,KapranovVoevodsky}.

\section{Conjugacy problem in $G_4^3$}

We begin with the conjugacy problem for the group $G_4^3$.

\subsection[Existence of the algorithmic solution]{The existence of the algorithmic solution of the conjugacy problem in the group $G_4^3$}

Let us denote the free groups on the generators $a_1, \dots, a_n$ by $F(a_1,\dots, a_n)$.
					
\begin{theorem}
Group $G_4^3$ with presentation 
$$\langle a,b,c,d \,|\, a^2=b^2=c^2=d^2 = (abcd)^2=(bcad)^2=(cabd)^2=1 \rangle,$$
has another presentation
$$\langle a,b,c \,|\, a^2=b^2=c^2=1 \rangle \freeprod_C (F(x,y,z) \leftthreetimes \langle d \,|\, d^2=1 \rangle)$$
where
$$C = \langle abc, bca, cab\rangle = F(x,y,z), \; x=abc, \; y = bca, \; z = cab,$$ and $$dxd=x^{-1}, dyd=y^{-1}, dzd=z^{-1}.$$
\end{theorem}

\begin{proof}
We will define two subgroups of $G_4^3$:		
$$A =  \langle a,b,c \,|\, a^2=b^2=c^2=1 \rangle,$$ $$B =  F(x,y,z) \leftthreetimes \langle d \,|\, d^2=1 \rangle.$$
The presentation of the group $A\freeprod_C B$ with $A$ and $B$ defined above is:
$$ \langle a,b,c,x,y,z,d \,|\, a^2=b^2=c^2=1, d^2=1,$$ $$dxd=x^{-1}, dyd=y^{-1}, dzd=z^{-1},
x=abc, y = bca, z = cab \rangle.$$
Replacing $x,y,z$ with $(abc), (bca), (cab)$ in the relations $dxd=x^{-1}, dyd=y^{-1}, dzd=z^{-1}$ we obtain the relations $$dabcd=cba, dbcad=acb, dcabd=bac.$$ 
Now we can eliminate $x,y,z$ and come to a presentation:
$$\langle a,b,c,d|a^2=b^2=c^2=1, d^2=1, dabcd=cba, dbcad=acb, dcabd=bac\rangle.$$
It is easy to see that this presentation is equivalent to the standard presentation $G_4^3$.
\end{proof}

The book \cite{LynSh} contains the following theorem: 

\begin{theorem} \label{th:LyndonSchupp}
Let $P=D\freeprod_R E$, and let $u=c_1\dots c_n$ be a cyclically reduced element of $P$ where $n \geq 2$. Then every cyclically reduced conjugate of $u$ can be obtained by cyclically permuting $c_1\dots c_n$ and then conjugating the result by an element of the amalgamated part $R.$	
\end{theorem}			
				
The main result of the present section is the following theorem:	

\begin{theorem} \label{thm:alg_conj_g43}
There exists an algorithm with input consisting of two words from $G_4^3 \setminus C$ which returns 1, if these words are conjugated in $G_4^3$, and returns 0, if not.
\end{theorem}

It turns out, that the conjugacy problem in the group $G_4^3$ is closely related to the so-called {\em twisted conjugacy problem}:

\begin{definition} \label{def:twisted_conjugacy}
Let $\phi$ be an automorphism of a group $F$. It is said that {\em the $\phi$-twisted conjugacy problem is solvable in the group $F$} if for any elements $u, v \in F$, we can algorithmically decide whether there exists such $g$ in $F$ that $\phi(g^{-1})ug = v$. For example, the $id$-twisted conjugacy problem is the standard conjugacy problem in $F$. 

Moreover, it is said that {\em the twisted conjugacy problem is solvable in the group $F$} if the $\phi$-twisted conjugacy problem is solvable for any $\phi \in {\text Aut}(F).$
\end{definition}
	
\begin{remark}
The generalised word problem for $C$ in $G_4^3$ is solvable. Also it is proved that $C\triangleleft G_4^3$.
\end{remark}

\begin{proof}[Proof of Theorem \ref{thm:alg_conj_g43}]
Consider two words $v,w\in G_4^3\setminus C$.

1. First let us suppose that one of those words belongs neither to the subgroup $A$, nor to the subgroup $B$. Then we can note by Theorem \ref{th:LyndonSchupp} that if they are conjugated, then some cyclic permutation of the first word will belong to the coset $wC$ of the second one. Indeed, conjugation by an element of $C$ cannot change the coset. Thus all suitable cyclic permutations of the first word should belong the coset $wC$ or, equivalently, these permutations and the second word should have the same image in the quotient group $G_4^3 / C$ under the natural homomorphism.

This quotient group is isomorphic to $(\Z_2 \oplus \Z_2) \ast \Z_2$, since $A/C \cong \Z_2 \oplus \Z_2$ and $B/C \cong \Z_2$. The word problem is obviously solvable in this group. Now let us have two words from one $C$-coset: $v$ and $w$. We need to determine whether they are conjugated in the group $G_4^3$.
 
 Note that if $h$ is in $C$ then $$hwh^{-1} = (ww^{-1})hwh^{-1} = w h^w h^{-1} = w \varphi_w(h) h^{-1},$$ where $\varphi_w\colon C\rightarrow C$ is an exterior  automorphism (since $w$ does not belong to $C$) sending $h$ to $h^w$. $C$ is a normal subgroup $C\triangleleft G_4^3$, so $\varphi$ is indeed an automorphism. Thus if $hwh^{-1} =v$, then $w \varphi_w(h) h^{-1} = v$, so $w^{-1} v =  \varphi_w(h) h^{-1}$
 
 Therefore we have reduced the conjugacy problem of the words $v$ and $w$ in the group $G_4^3$ to the following question: whether there exists an element $h$ of the subgroup $C$ such that $w^{-1} v = \varphi_w(h) h^{-1}$. In other words, we need to solve a twisted conjugacy problem in the subgroup $C$.
	 
Note also that the group $C$ is a free group of finite rank (in fact, it is a group of rank 3). It is shown in \cite{Bogo} that the twisted conjugacy problem is solvable in finitely generated free groups. Therefore, in this case the conjugacy problem in the group $G_4^3$ is solvable.

2. Now let us consider the case, when at least one of the words $v,w$ lies either in the subgroup $A$ or the subgroup $B$. Without loss of generality let $v$ be a word from $A\setminus C$. Denote the cyclically reduced word for $v$ by $\tilde{v}$. A conjugate to the word $\tilde{v}$ is of the form $b_n^{-1}a_n^{-1}\dots b_1^{-1} a_1^{-1} \tilde{v} a_1 b_1 \dots a_n b_n$, where $a_i\in A, b_i\in B$. Suppose this word is cyclically reduced and $b_1 \neq 1$. 

Consider the word $\hat{v} = a_1^{-1}\tilde{v} a_1$. If $\hat{v}$ is an element of $C$, then $\tilde{v}$ is an element of $C$ as well, since $C$ is a normal subgroup of $G_4^3$. But that contradicts the assumption that $v\in A\setminus C$. Therefore $\hat{v}\notin C$. That means that we can cyclically reduce the word $b_n^{-1}a_n^{-1}\dots b_1^{-1} a_1^{-1} \tilde{v} a_1 b_1 \dots a_n b_n$ obtaining the word $\hat{v} = a_1^{-1}\tilde{v} a_1 \in A\setminus C$. That means that cyclically reduced elements conjugated to $v$ can only be elements of $A$ and conjugating elements have to belong to $A$ as well. 

Using the same arguments we deduce that cyclically reduced elements conjugated to an element of $B$ can only be elements of $B$ and conjugating elements have to belong to $B$ as well. That allows us to reduce the conjugacy problem for words from $A$ and $B$ in $G_4^3$ to the conjugacy problem in $A$ and $B$ correspondingly. It is obvious that the conjugacy problem in the subgroups $A$ and $B$ is solvable. 
\end{proof}

We proved that the conjugacy problem is solvable in the group $G_4^3$. We will go a bit further though, presenting in the next section the explicit algorithm of the conjugacy problem solution.
	 
\subsection{Algorithm of solving the conjugacy problem in $G_4^3$}

First, we present an algorithm solving the twisted conjugacy problem in the subgroup $C$. \\
	 
The input of the algorithm is the set $ \{u,v,\phi\}$, where $u,v$ are words from $C$, and $\phi$ is an automorphism of $C$.

We already have a free basis for $C$ and, adding a new letter $z$, we get a free basis for $C' = C\ast \langle z \rangle$. We will define some objects to help ourselves.

\begin{definition}
We say that the {\em extension of $\phi$}, which is denoted by $\phi' \in Aut(C')$, is defined by the formula $\phi'(z) = uzu^{-1}$, $\phi'(c)=\phi(c)$ for $c\in C$.
\end{definition}

Let us denote by $\gamma_{g}$ the inner automorphism of $C'$ defined by the formula $\gamma_{g}(h) = g^{-1}hg$ for any $h \in C'$. Furthermore, we introduce the notation $\gamma_g(\phi)$ for the composition of $\gamma_g$ and the extension $\phi'$ by the formula $$\gamma_g(\phi)(h)=\gamma_g(\phi'(h)).$$ $Fix(\gamma_{v}(\phi))$ denotes the set of fixed points of that mapping.

\begin{lemma}
The words $u$ and $v$ are $\phi$-twisted conjugated if and only if $Fix(\gamma_{v}(\phi))$ contains an element of the form $g^{-1}zg $ for some $g\in C$ (and, in this case, $g$ is a valid twisted conjugating element). 
\end{lemma}

\begin{proof} 
In fact, suppose that $v = (\phi(g))^{-1}ug$ for some $g\in F$. A simple calculation shows that $g^{-1}zg$ is then fixed by $\gamma_{v}(\phi)$. Conversely, if $g^{-1}zg$ is fixed by $\gamma_{v}(\phi)$ for some $g\in F$, then $gv^{-1}(\phi(g))^{-1} u$ commutes with $z$.  And this implies $gv^{-1}(\phi(g))^{-1} u  =  1$, since this word contains no occurrences of $z$ and we work in the group $C'=C\ast \langle z \rangle$. Hence, $v = (\phi(g))^{-1}ug$ so $u$ and $v$ are $\phi$-twisted conjugated (with $g$ being a twisted conjugating element). 
\end{proof}

We can algorithmically find a generating set for $Fix(\gamma_{v}(\phi))$, we can also decide if this subgroup contains an element of the form $g^{-1}zg$ for some $g\in F$. One can, for example, look at the corresponding (finite) core graph for $Fix(\gamma_{v}(\phi))$ (algorithmically computable from a set of generators) and see if there is some loop labelled $z$ at some vertex connected to the base-point by a path not using the letter $z$. And, if this is the case, the label of such a path provides the $g$, i.e. the required twisted conjugating element. \\
	 
Now we have all the necessary instruments to describe an algorithm of solving conjugacy problem for words in $G_4^3$.
	 
The input of algorithm is the set $ \{w,v\}$, where $w,v$ are words from $G_4^3.$ If either of those words lies in the subgroup $A\cup B$, then the algorithmic solution is obvious. Now let the words $w,v \in G_4^3\setminus (A\cup B)$. In that case the algorithm is as follows.

\begin{enumerate}
\item Write down all cyclic permutations of $w$.

\item For any permutation $\widehat{w}$ check if it is in the $C-$coset $vC$, using the algorithm of solving the word problem in $G_4^3/C \cong (\Z_2 \oplus \Z_2) \ast \Z_2$. If no, this permutation does not fit. If yes, go to step 3.

\item For permutation $\widehat{w}$ we define the automorphism $\varphi_{\widehat{w}}\in Aut(C)$, defined this way: $\varphi_{\widehat{w}}: h\rightarrow h^{\widehat{w}}$.

\item Define the group $F = C \ast \langle z \rangle$, define the automorphism $\psi \in Aut(F)$ this way: its restriction to $C$ is $\varphi_{\widehat{w}}$ and $\psi(z)=z$.

\item For $\varphi_{\widehat{w}v{-1}}\circ \psi \in Aut(F)$ construct finite basis of group $Fix(\varphi_{\widehat{w}v{-1}}\circ \psi)$.

\item Having this finite basis of $Fix(\varphi_{\widehat{w}v{-1}}\circ \psi)$ construct its core-graph .

\item Check if there is a vertex with loop labeled $z$ connected with marked vertex with a path without $z$-labels.

\item If there is one then exists a word $h$ in $C$ with $w^{-1} v = 
\varphi_{\widehat{w}}(h) e h^{-1}$ where $e$ -- empty word, so empty word and $w^{-1} v$ are $\varphi_{\widehat{w}}$-conjugated. In this way $\widehat{w}$ and $v$ are also conjugated.

\item Otherwise, there is no such word.
\end{enumerate}

\section{The word problem for $G_5^4$}

Consider the group $G_5^4$ and its subgroup $H_4$ (see Section \ref{sec:h_k_and_geomtery}). As was remarked before in Section \ref{sec:h_k_and_geomtery}, if the word problem is solvable in the subgroup $H_k$, it is solvable in the group $G_{k+1}^k$ as well. Therefore we need to study the group $H_4$ to understand the solvability of the word problem in the group $G_5^4$. 

\subsection{Presentation of the group $H_4$}

The following lemma is central in the proof of algorithmic solvability of the word problem in the group $G_5^4$.

\begin{lemma} 
	The subgroup $H_4\subset G_5^4$ has a presentation 
	\begin{align*}
		\langle a,b,c,d \,|\, a^2=b^2=c^2=d^2=1, \, [(ab)^2,(cd)^2]&=1, \\
		[(ac)^2,(bd)^2]=1, \, [(ad)^2,(bc)^2]&=1\rangle.
	\end{align*}
	\label{lem:h_4}
\end{lemma} 

\begin{proof}
Consider five points in $\R{}P^3$. Consider a motion of 5 points and forbid the situation when three points lie on the same straight line. 

Now, since we study the group $H_4$, we may consider a motion of the following type: four point (say, $1,2,3,4$) are in general position and fixed, and the fifth point 5 moves around, see Fig.~\ref{tetrahedron}. Generators of the group $H_4$ correspond to the moments when some four points lie in the same plane. Therefore they are in one-to-one correspondence with triples of points defining the plane where in that moment lies the moving point 5. For example, the triple $(123)$ corresponds to the moment, when the points $1,2,3,5$ lie in the same plane. We may say that the point 5 moves in the manifold $M=\R{}P^3\setminus \{\hbox{tetrahedron}\,\, (1234)\}$. Here saying ``tetrahedron'' we mean the configuration of $\R{}P^1$.

\begin{figure}[t!]
\centering\includegraphics[width=350pt]{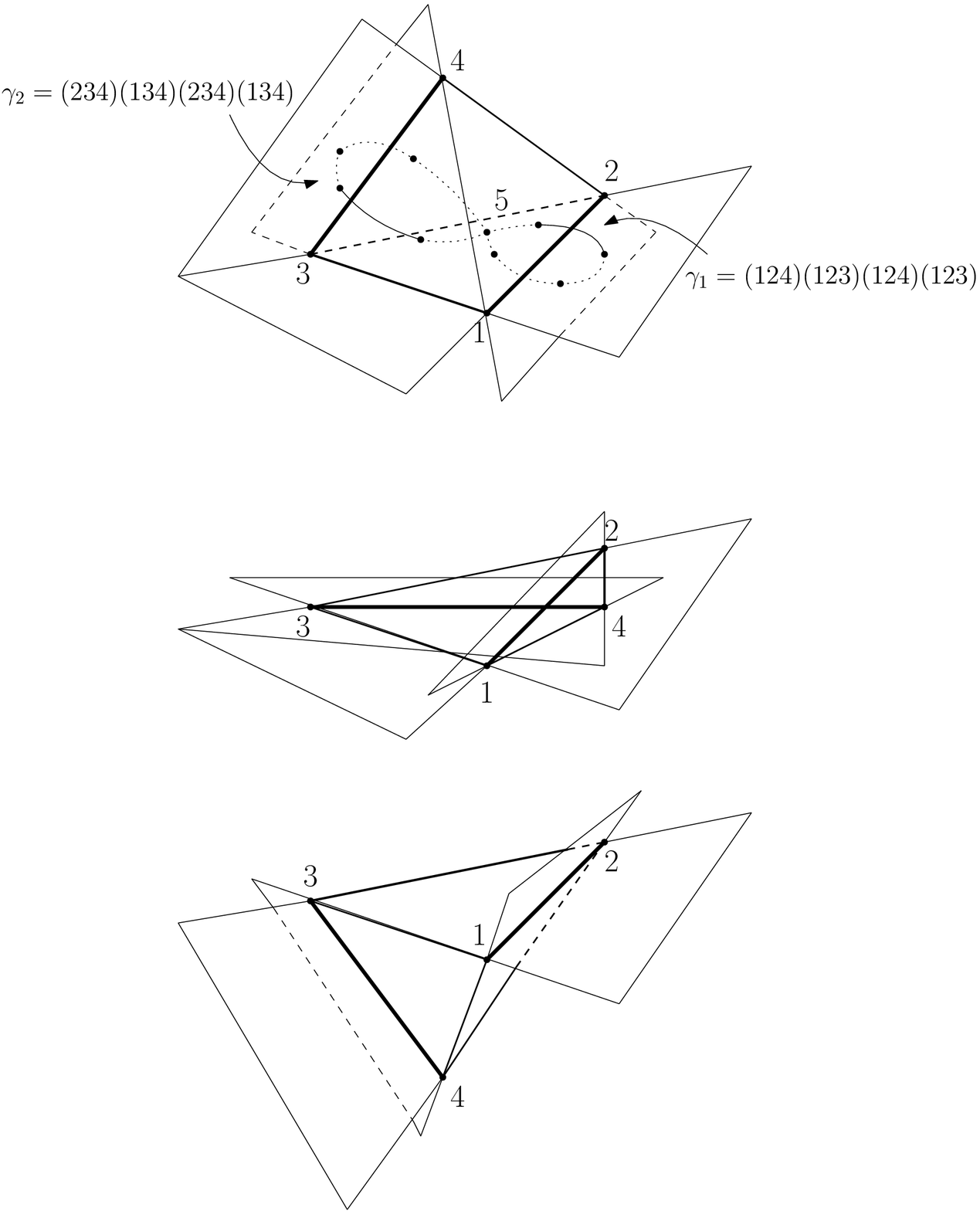}
\caption{The point 5 moves along the loops $\gamma_1$ and $\gamma_2$; when the point 4 moves, the tetrahedron becomes inverted}
\label{tetrahedron}
\end{figure}

Consider two loops $\gamma_1,\gamma_2$ depicted in Fig.~\ref{tetrahedron}, upper part. The fundamental group of the space $M$ is formed by such paths looping around the edges of the tetrahedron.

Now let us allow the points $1,2,3,4$ move. There is a finite number of singular moments when the four points lie on a plane and the tetrahedron transforms into a quadrilateral with crossing diagonals, see Fig.~\ref{tetrahedron}, middle part. Note that the tetrahedron lies in a projective space, not in Euclidian $\mathbb{R}^3$. At that point the forbidden set, which we cut from the manifold in which the point 5 moves, changes and becomes singular, denote the resulting manifold by $M_{\text{sing}}$.

If the points continue moving, the tetrahedron becomes ``inverted'', see Fig.~\ref{tetrahedron}, lower part. Thus two copies of the manifold $M$ become glued along the singular moment. This glueing may be imagined in the following way. Let us take the manifold $M$ and deform it into the manifold $M_{\text{sing}}$. We obtain a submanifold in $M\times [0,1]$ such that the manifold $M_{\text{sing}}$ is embedded into its boundary. Then we perform the same operation with the second copy of the manifold $M$ and glue the resulting manifolds along the submanifold $M_{\text{sing}}$.  

Now we can describe the space in which the point 5 moves: it is a manifold $\widehat{M}$ obtained by glueing together copies of the manifold $M$ along all singular moments of the movement of the point 5 as described above. Fig. ~\ref{tetrahedron} depicts two copies of manifold $M$ (upper and lower parts of the figure) and the manifold $M_{\text{sing}}$ along which they are glued (middle part of the figure).

Each singular moment produces a relation in the group $H_4$. Before the singular moment $t$ we had two loops: $\gamma_1$ and $\gamma_2$. After the singular moment $t$ we obtain two new loops $\gamma_1', \gamma_2'$, which go around the same edges of the tetrahedron as the loops $\gamma_1,\gamma_2$ (the edges $(12)$ and $(34)$ in the example in Fig.\ref{tetrahedron}). We have the following equalities: $$\gamma_1'=\gamma_1, \;\;\; \gamma_2'=\gamma_1^{-1}\gamma_2\gamma_1.$$ We consider the fundamental group of the manifold $\widehat{M}$. To obtain it, at each singular moment $t$ we glue two fundamental groups of the copies of the manifold $M$ and obtain the equality $\gamma_2'=\gamma_2$. Therefore the loops commute and we get the relation relation $$(124)(123)(124)(123)(234)(134)(234)(134)=$$ $$(234)(134)(234)(134)(124)(123)(124)(123).$$

Introducing the notation $a=(123), b=(124), c=(134), d=(234)$ we formulate the relation $$(baba)(dcdc)=(dcdc)(baba).$$

The relations of the form $a^2=b^2=c^2=d^2=1$ are evident. 

Taking that into account, we obtain the relations of the form $$(abab)(cdcd)=(cdcd)(abab).$$

Considering the loops around other pairs of opposite edges, we obtain the rest of the relations. Since there are no more singular configurations of the points movement, there are no other relations. The lemma is proved.
\end{proof}

\subsection{The Howie diagrams}

To study the group $H_4$ and the word problem in it we will need the so-called {\em Howie diagrams} --- a geometric way of representing group elements, in some aspects resembling the van Kampen diagrams (see, for example, \cite{howie}, where Howie diagrams are called {\em relative diagrams}). Just like in the case of van Kampen diagrams, one can consider disc diagrams, spherical diagrams, etc. In the present section we will only need disc Howie diagrams. 

To be more precise, the Howie diagrams are defined in the following way. First we define a {\em relative group presentation}. 

\begin{definition}
A {\em relative group presentation} is an expression of the form $\mathcal{P} = \langle G, x \,|\, R\rangle$, where $G$ is a group, $x$ is a set disjoint from $G$, and $R$ is a set of cyclically reduced words in the free product $G \ast \langle x \rangle$.

\label{def:relative_presentation}
\end{definition}

This notion was introduced by Bogley and Pride, see \cite{bogley}. Now let us fix a group $G$ and let the set $x$ be one-element. Consider a map $M$ on the plane and let the edges of $M$ (its 1-cells) be oriented. We will treat the faces (2-cells) of the map as polygons, and for each face $F$ saying a {\em corner of the face} we mean an intersection of a sufficiently small neighbourhood of a vertex of the face with the face $F$ itself.

Let us decorate every edge of the map with the letter $x$. Furthermore, we decorate every corner of every face of the map with a word $w\in G$. Now we can define {\em face-labels} and {\em vertex-labels} of the map $M$.

\begin{definition}
	Let $M$ be a map as described above. A {\em vertex-label} of a vertex $v$ is a word $\nu\in G$ obtained by reading all the corner labels while walking around the vertex $v$ in the clockwise direction.
	
	A {\em face-label} of a face $F$ of the map $M$ is a word $\nu \in G \ast \langle x \rangle$ obtained by reading all corner labels and edge labels while walking around the boundary of the face $F$ in the counterclockwise direction. As usual, if the orientation of an edge is compatible with the counterclockwise direction convention, we shall read $x$, otherwise we read $x^{-1}$.
	\label{def:diagram_labels}
\end{definition}

Now we have all the tools to define the Howie diagram.

\begin{definition}
	Let $K = \langle G,x \,|\, R\rangle$ be a relative group presentation with the set $x$ being one-element. A decorated map $M$ as described above is called a {\em Howie diagram} over that relative presentation, if the following conditions are satisfied:
	
	\begin{enumerate}
		\item for each interior vertex of the map its vertex-label is trivial in the group $G$;
		\item for every interior 2-cell $F$ of the map its face-label $w$ should lie in the set $R$ of the relations of the relative presentation $\langle G, x \,|\, R\rangle$.
	\end{enumerate}
	\label{def:howie_diagram}
\end{definition}

\begin{remark}
	Note that in \cite{howie} the relative diagrams are considered on the sphere and for that reason the corresponding maps have one vertex distinguished. Here we work with planar diagrams and may treat all vertices equally.
\end{remark}

The Howie diagrams are useful because the following analogue of the van Kampen lemma holds for them.

\begin{lemma}
	Let $H$ be a group, $t$ be a letter, and $G$ be a group given by the relative presentation $$G=\langle H,t \,|\, R\rangle.$$
	Let furthermore the natural mapping $\iota\colon H\to G$ be injective. Then for any word $w\in H\ast \langle t \rangle$ its image $\iota(w)$ is trivial in $G$ if and only if there exists a Howie diagram with single exterior cell over this presentation such that its exterior cell face-label is the word $w$.
	\label{lem:howie_van_kampen}
\end{lemma}

\subsection{The solution to the word problem in $H_4$}

Now we shall use the Howie diagrams techniques and, in particular, Lemma \ref{lem:howie_van_kampen} to tackle the word problem in $H_4$. First, consider a subgroup $H_3\subset H_4$ generated by the letters $\{a,b,c\}$. It is isomorphic to $\langle a\rangle_2\ast\langle b\rangle_2\ast\langle c\rangle_2$. The group $H_4$ may be given by the following relative presentation: $$H_4=\langle H_3, d \,|\, R\rangle,$$ where the set of relations $R$ is naturally obtained from the presentation of $H_4$ given in Lemma \ref{lem:h_4}:
$$R = \{d^2=1, \;\; [(ab)^2,(cd)^2]=1, \;\; [(ac)^2,(bd)^2]=1, \;\; [(ad)^2,(bc)^2]=1\},$$
where the square brackets denote the usual commutator. By cyclic permutations we can transform those relation into such a form, that the letter $d$ is at the end of every relation:

\begin{enumerate}
	\item $d^2=1,$ 
	\item $(cababc)d(c)d(baba)d(c)d =1,$
	\item $(bacacb)d(b)d(caca)d(b)d=1,$
	\item $(acbcba)d(a)d(bcbc)d(a)d=1.$
\end{enumerate}

The inclusion mapping $\iota\colon H_3\to H_4$ is injective, so Lemma \ref{lem:howie_van_kampen} holds for this relative presentation. Therefore, a word $w$ is trivial in the group $H_4$ if and only if there exists a Howie diagram with single exterior cell (that is, a connected Howie diagram) with the word $w$ appearing as the face-label of the exterior face of the diagram. Let us denote this diagram by $D$.

Now we transform the diagram $D$ in the following manner. First, we reduce all the bigons corresponding to the relations $d^2=1$ replacing each bigon with vertices $v_1, v_2$  with an edge connecting the same vertices and oriented arbitrarily. After the removal of bigons, we ``forget'' the orientation of the edges. This operation is legal due to the relation $d^2=1$. Finally, we forget the label on the edges of the diagram, since all of them are equal to $d$. This operation does not change the vertex-labels, but transforms the face-labels. We will come back to that change later.

Thus we obtain an unoriented diagram $D'$ without bigonal cells. Furthermore, every relation of the set $R$ (except for the $d^2=1$ relations, which are no longer present in the obtained unoriented diagram) include exactly four subwords from the group $H_3$. That means that every interior face of the diagram $D'$ has four vertices. Note that the diagram $D'$ may not be homotopy equivalent to a disc. Nevertheless, it is connected and homeomorphic to a tree, some vertices of which are ``blown up'' into a submap homeomorphic to a disc.

Each cell of the map $D'$ corresponds to a relation of the set $R$. Moreover, there is no bigonal cells. That means, that the diagram has only cells depicted in Fig.~\ref{fig:types_of_cells}. We shall call them the {\em types} of the cells of the diagram.

\begin{figure}[t!]
\centering\includegraphics[width=300pt]{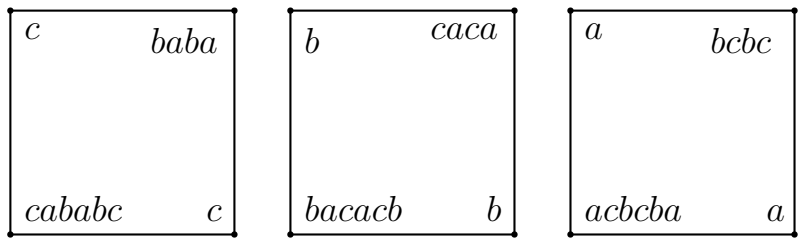}
\caption{Three types of cells of the diagram $D'$}
\label{fig:types_of_cells}
\end{figure}

Whenever two cells are glued together by two or more edges, at least one vertex becomes interior. By definition of the Howie diagram, the vertex-label of that vertex must be trivial in the group $H_3$. Consider the leftmost part of Fig.~\ref{fig:types_of_cells} which corresponds to the second relation of the set $R$. There are three different labels corresponding to a cell of that type: $c, baba$ and $cababc$. Among their pairwise concatenations there is only one word, trivial in $H_3$: the word $c^2$. Therefore, there are only two ways to glue together two cells of this type by at least two edges. Those situations are depicted in Fig.~\ref{fig:glueing_of_cells}. Analogous operations may be performed for two other types of cells. At the same time, no two cells of different types may be glued by two or more edges, because no trivial in $H_3$ vertex-label may be obtained this way.

\begin{figure}[t!]
\centering\includegraphics[width=300pt]{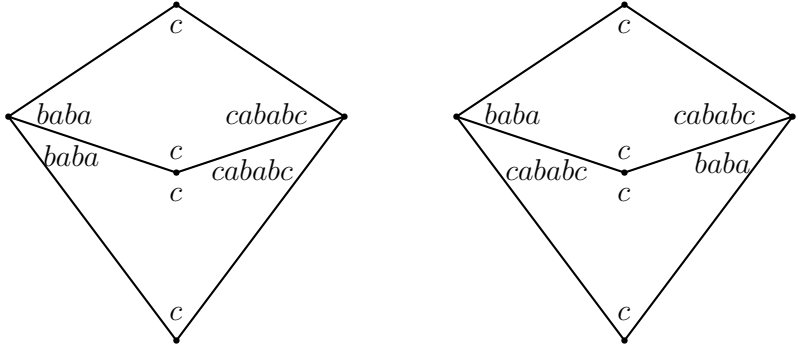}
\caption{Possible ways to glue together two cells of the first type by two or more edges}
\label{fig:glueing_of_cells}
\end{figure}

Now let us consider an interior vertex of the diagram $D'$.

\begin{statement}
	For every interior vertex $v$ of the diagram $D'$ there are either two or at least four cells of the diagram incident to that vertex.
\end{statement}

\begin{proof}
	The case of a vertex incident to exactly two cells was studied above. There are six possible glueings of two cells producing such vertices (two possible ways for each of three types of cells). Now let us prove that all other interior vertices are incident to at least four cells.
	
	Suppose there exists an interior vertex $v$ incident to exactly three cells. Then the vertex-label $l$ of vertex $v$ is of the form $l = l_1 l_2 l_3$, where each $l_i$ is from the set $W=\{a, b, c, baba, caca, bcbc, cababc, bacacb, acbcba\}$. The label $l$ must be trivial in the group $H_3$. Without the loss of generality we may say that it means that $l_1^{-1}=l_2 l_3$. It is easy to check that for all possible pairs of words of this set their product does not lie in the set $W$. Therefore the label $l$ cannot be trivial and there is no such vertex $v$.
\end{proof}

Let us apply one more transformation to the diagram $D'$. For each pair of cells glued together by two edges we replace them with one cell, deleting the two edges by which the cells are glued and replacing the labels in the corners of the cells with their products, see Fig.~\ref{fig:simplification}. We call this operation a {\em contraction} of cells.

\begin{figure}[t!]
\centering\includegraphics[width=300pt]{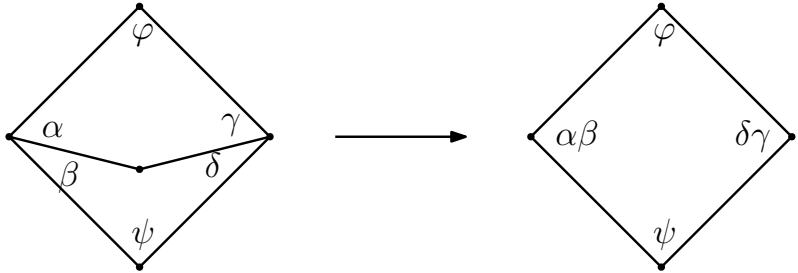}
\caption{Transforming two cells glued by two edges into one}
\label{fig:simplification}
\end{figure}

After this transformation is applied to all suitable pairs of cells, we obtain a decorated map, denote it by $M$. Note that now the face-labels are not necessary words from the set of relations $R$ and therefore the map may be not a Howie diagram. On the other hand, it is a map such that its every interior vertex is incident to at least four cells and every two cells are glued together by at most one edge. In other words, the map $M$ is a $(4,4)$-map in terms of the small cancellation theory.

Every cell of the map $M$ is obtained by contracting a chain of the cells of one type of the diagram $D'$. We extend the notation and say that a cell $F$ obtained by contraction of a chain of cells of type $i$ is of type $i$. Let us describe what kind of labels the cells of the map $M$ may have in their corners. Suppose that a cell $F$ is of the first type (other types are treated analogously). Then two opposite corners of the new cell contain the letter $c$ and the other two corners contain some words from the group $\langle cababc \rangle \ast \langle baba \rangle$. Moreover, if we know the label $u$ in a corner of this cell, we can uniquely determine the label $u'$ in the opposite corner. Indeed, if $u=c$, then $u'=c$ as well. If $u\in \langle cababc \rangle \ast \langle baba \rangle$, the opposite label may be obtained by replacing all occurrences of the word $baba$ in the label $u$ with $cababc$ and vice versa. 

To prove the main statement of the present subsection we shall use the following theorem (see \cite{LynSh}):

\begin{theorem}
	Let $M$ be a connected $(4,4)$-map. Then the number of 2-cells of the map does not exceed $$\frac{1}{4}\left(\sum_{F\subset M}[4-i(F)]^2\right),$$ where the sum is taken over all 2-cells of the map and $i(F)$ denotes the number of 2-cells incident to the cell $F$ (i.e. the cells with a common edge with $F$).
	\label{thm:lyndon}
\end{theorem} 

Note that every interior cell of the map $M$ has at least 4 adjacent cells and thus does not yield a positive summand in that sum. Let us denote the cells with common edges with the outer boundary of the map by $D_i$ and their number by $N$. Each edge of the outer boundary of the map corresponds to a letter $d$ in the word $w$. Therefore the number $N$ cannot be greater than the number of letters $d$ in the word $w$. Thus the total number of the cells of the map $M$ cannot exceed $\frac{1}{4}(4|w|_d)^2=4|w|_d^2,$ where $|w|_d$ denotes the number of letters $d$ in the word $w$.

We have demonstrated that if a word $w\in H_4$ is trivial, then there exists a connected $(4,4)$-map $M$ decorated as described above such that the face-label of its exterior cell is the word $w$. As follows from Theorem \ref{thm:lyndon} the total number maps (without labels) that can be potentially transformed by labelling into the map $M$ with the word $w$ on its outer boundary is finite. Now we shall prove the following statement:

\begin{lemma}
	Let $w$ be a word in $H_4$ and let $M$ be a $(4,4)$-map homeomorphic to a disc with the number of cells not exceeding $4|w|_d^2$. Then there exists a finite constructive algorithm determining whether there is such a labelling of the corners of the cells of $M$ that vertex-labels of all interior vertices of the maps were trivial and the face-label of the exterior cell of the map $M$ was exactly $w$.
	\label{lem:disc_case}
\end{lemma}

\begin{proof}
	The proof shall be by induction on the number of cells of the map. The general strategy is as follows: we shall change the word $w$ and the map $M$ by finding relations inside the word $w$ and removing them in such a way, that the number of cells in $M$ decreases.
	
	Let us call two edges of a map {\em neighbouring} if they have a common vertex. The map $M$ is finite, hence it has at least one cell with two neighbouring edges incident to the exterior cell of the map (such edges should also be called {\em exterior}). Let us take such a cell, denote it by $F$, and consider the label $l$ in the corner between the two neighbouring exterior edges. Let us consider the possible values of the label $l$. Without the loss of generality we should suppose that the cell $F$ is of the first type. \\
	
	1. $l$ could be a word from the group $\langle cababc \rangle \ast \langle baba \rangle$. Note that since $l$ is the label in the corner between two exterior edges, the word $w$ has a subword of the form $dld$. Therefore there is finite number of possible values of $l$. For each of them we can uniquely determine labels in all other corners of the cell $F$. Denote the label in the opposite corner by $l'$. Now we can change the word $w$ by replacing the subword $dld$ with a word $dldcdl'dc$ and get a new map $M'$ with fewer cells (effectively we contract the cell $F$ with the exterior cell of the map $M$ in the same manner as we contracted interior cells earlier). \\

	2. If $l=c$, then in the neighbouring corner the cell $F$ has a label $u\in \langle cababc \rangle \ast \langle baba \rangle$. The word $u$ uniquely determines the label $u'$ in the opposite corner. To perform the contraction of the cell $F$ with the exterior cell we need to prove that there are finitely many possible values of $u$.
	
	The corner with the label $u$ is incident to a a vertex which we shall denote by $v$. Only finite number of cells are incident to the vertex $v$, let us denote their number by $K$ (this number may easily be bounded considering the total number of cells of the map $M$). Note that $v$ is an exterior vertex, because it is an endpoint of an exterior edge. Therefore, the vertex-label of $v$ is a subword of the word $w$ and hence it cannot be longer than the word $w$. Now suppose that the length of the label $u$ is greater than $|w|+K$.
	
	The labels in the corners of the cells incident to the vertex $v$ may either equal $a, b$ or $c$ or lie in one of the groups $\langle cababc \rangle \ast \langle baba \rangle, \langle bacacb \rangle \ast \langle caca \rangle, \langle acbbca \rangle \ast \langle bcbc \rangle$ depending on the type of the corresponding cell. It is easy to see that with the addition of each label the length of $u$ may decrease at most by one. That means that $|l(v)|\ge |u| - (K-1) > |w|+K-K+1 = |w|+1$. That is impossible, therefore the length of $u$ cannot be greater than $|w|+K$.
	
	We proved that there is only finite number of possible values of $u$. For each of them we can uniquely determine the value of $u'$ and contract the cell $F$ with the exterior cell, obtaining a map with fewer cells. \\
	
	There are no other possibilities for the values of the label $l$. That means that if we take a cell with two neighbouring exterior edges, we may obtain finitely many maps with fewer cells. By induction we either find the decomposition of the word $w$ into a product of generators or prove that there is no required Howie diagram.
\end{proof}

Now we are ready to prove the main theorem of this section.

\begin{theorem}
	The word problem is algorithmically solvable in the group $H_4$.
\end{theorem}

\begin{proof}
	Consider a word $w\in H_4$. As we have shown, it is trivial in $H_4$ if and only if there exists a $(4,4)$-map with labels in the corners of its cells such that the vertex-label of each interior vertex of the map is trivial in the group and the face-label of the exterior cell is the word $w$.
	
	As we noticed before, this map is a tree with a disc submap at each vertex. The total number of cells in the map is not greater than $4|w|_d^2$, the number of the edges of the tree is bounded by the same number. Therefore the number of such trees is finite.
	
	Now for each disc submap we may perform the reduction algorithm from Lemma \ref{lem:disc_case}. Reducing each disc submap in this manner we come down to a tree and some new word $\tilde{w}$. Letters $d$ break the word $\tilde{w}$ into a finite sequence of subwords $w_1, \dots, w_n$. To check whether the map admits a labelling giving the word $\tilde{w}$, we need to choose a starting vertex (that can be done in a finite number of ways) and place the labels $w_1, \dots, w_n$ into the vertices of the tree walking around it starting from the chosen vertex (when we reach a leaf of the tree, we shall go back along the same edge). That gives us a finite algorithm of determining whether the necessary labelling exists.
	
	Hence, we have constructed an algorithm which takes a word $w$ and returns ``true'' if there exists a labelled map with the word $w$ as the label of its outer cell, or ``false'' otherwise. Therefore the word problem in the group $H_4$ is solvable.
\end{proof}

As was remarked in Section \ref{sec:h_k_and_geomtery}, the solvability of the word problem in $H_k$ yields the solvability of the word problem in $G_{k+1}^k$. Thus we obtain the following theorem.

\begin{theorem}
	The word problem is algorithmically solvable in the group $G_{5}^4$.
\end{theorem}

\section{Further directions}

In the present paper we only considered the cases of the groups $G_4^3$ and $G_5^4$. A natural question is to consider all groups $G_{k+1}^k$ and tackle the word and conjugacy problems in those groups. The possible way to approach those problems is the following.

We have noted (see Section \ref{sec:h_k_and_geomtery}) that the word problem in the group $G_{k+1}^k$ essentially boils down to the same problem in the $H_k$ subgroup. We used this consideration to solve the word problem in $G_5^4$ by proving the existence of a certain presentation of the group $H_4$ which was useful for our purposes. It seems that the following conjecture should be true:

\begin{conjecture}
	For any group $G_{k+1}^k$ its subgroup $H_k$ admits the following presentation:
	
	$$H_k=\langle b_1,\dots, b_k \,|\, b_1^2=\dots=b_k^2=1, [(b_i b_j)^2, (b_m b_l)^2]=1\rangle$$
	
	where $\{i,j,m,l\}$ go over all possible sets of distinct numbers from $\{1,\dots, k\}$.
\end{conjecture}

The possible ways of proving this conjecture include both geometric arguments (as was done for the case of $H_4$ in the present paper) and algebraic considerations. If this conjecture holds, the word problem in $H_k$ (and, therefore, in $G_{k+1}^k$) should be solvable by the same techniques as were developed in the present section for $k=4$.

\end{document}